\documentclass[11pt]{amsart}

\usepackage[T1]{fontenc}
\usepackage{lmodern}
\usepackage{amsmath,amssymb,amsthm,mathtools, fullpage}
\usepackage{booktabs}
\usepackage{enumitem}
\usepackage[hidelinks]{hyperref}
\usepackage[nameinlink,noabbrev]{cleveref}
\usepackage{microtype}

\newtheorem{theorem}{Theorem}[section]
\newtheorem{proposition}[theorem]{Proposition}
\newtheorem{lemma}[theorem]{Lemma}
\newtheorem{corollary}[theorem]{Corollary}

\theoremstyle{definition}
\newtheorem{definition}[theorem]{Definition}
\newtheorem{remark}[theorem]{Remark}

\newcommand{\kk}{\mathbb K}
\newcommand{\En}{\mathcal E_n}
\newcommand{\wt}{\operatorname{wt}}
\newcommand{\inop}{\operatorname{in}}
\newcommand{\Cut}{\operatorname{Cut}}
\newcommand{\esort}{\operatorname{esort}}
\numberwithin{equation}{section}

\title[Quadratic Gr\"obner bases for cycles and ring graphs]
{Quadratic Gr\"obner Bases for Cut Ideals of Cycles and\\ Ring Graphs}

\author{Hidefumi Ohsugi}
\address{Hidefumi Ohsugi,
Department of Mathematical Sciences,
School of Science,
Kwansei Gakuin University,
Sanda, Hyogo 669-1330, Japan}
\email{ohsugi@kwansei.ac.jp}

\date{}

\begin{document}

\begin{abstract}
Let $C_n$ be the cycle of length $n\ge3$ and let $I_{C_n}$ be its cut ideal.
We show that $I_{C_n}$ has a quadratic Gr\"obner basis with respect to an explicit weight order.
Since the defining configuration consists of $(0,1)$-vectors, the initial monomials of such a basis are automatically squarefree.
As the cut polytope of a cycle is the parity polytope, the result gives a regular unimodular flag triangulation of this classical polytope.
Together with the known tree case and the clique-sum theorem for cut ideals,
the cycle result also yields a quadratic Gr\"obner basis for the cut ideal
of every connected ring graph with at least one edge,
thereby supplying the missing cycle input and establishing the result for connected ring graphs.
\end{abstract}

\maketitle

\section{Introduction}

For a finite simple graph $G=(V,E)$ and a subset $S\subseteq V$, the \emph{cut} determined by $S$ is
\[
 \Cut_G(S):=\{uv\in E: |\{u,v\}\cap S|=1\}.
\]
Thus $\Cut_G(S)=\Cut_G(V\setminus S)$. Cut ideals were introduced by Sturmfels and Sullivant \cite{SturmfelsSullivant2008} in connection with algebraic statistics. 
Engstr\"om \cite{Engstrom2011} proved that the cut ideal $I_G$
of a graph $G$ is generated by quadratic binomials if and only if
$G$ is $K_4$-minor-free.
Shibata \cite[Corollary~2.4]{Shibata2015} proved that $I_G$
has a quadratic Gr\"obner basis if $G$ has no $K_4$- or $C_5$-minor.
In particular, the cut ideal of a cycle is generated by
(homogeneous) quadratic binomials, but Shibata's result does not
apply to $C_n$ for $n\ge5$, since $C_n$ has $C_5$ as a minor.

The cycle case also identifies the present problem with several objects studied in other areas. 
In fact, the cut vectors of $C_n$ are precisely the vectors in $\{0,1\}^n$ of even Hamming weight. Hence the cut polytope of $C_n$ is the parity polytope
\[
 P_n^{\mathrm{even}}=\operatorname{conv}\{x\in\{0,1\}^n:\textstyle\sum_{i=1}^n x_i\equiv0\pmod2\}.
\]
This is the codeword polytope of the single parity-check code and appears naturally in linear-programming decoding of binary error-correcting codes.
See, for example, \cite{BarmanLiuDraperRecht2013}. 
Under the affine change of coordinates $x_i\mapsto 1-2x_i$,
the same polytope is the half cube (or demihypercube).
See, for example,  \cite{Green2009}.
Moreover, up to a relabeling of variables, its toric ideal is the $\mathbb Z_2$ group-based phylogenetic ideal on a claw tree, as used below through the connection of Nagel and Petrovi\'c \cite[Lemma~3.1]{NagelPetrovic2009}. Thus even in the cycle case the same toric ideal sits naturally in polyhedral combinatorics, coding theory, and algebraic statistics.

A second reason to study the cycle case separately comes from ring graphs.
Following Gitler, Reyes, and Villarreal, a ring graph is a graph
in which each block that is neither a bridge nor a vertex can be
built from a cycle by successively attaching paths of length at least
two that meet the graph already constructed only at their endpoints,
which are adjacent
\cite[Definition~2.6]{GitlerReyesVillarreal2010}.
Equivalently, a connected ring graph can be obtained from trees and cycles by repeated clique sums along a vertex or an edge.
See \cite[Definition~6.1 and the discussion following it]{NagelPetrovic2009}. 
Ring graphs form a natural subclass of the $K_4$-minor-free graphs
and contain all outerplanar graphs
\cite[Theorem~2.13 and Proposition~2.17]{GitlerReyesVillarreal2010}.
Sturmfels and Sullivant showed that Gr\"obner bases of cut ideals can be combined under clique sums along cliques of size at most three \cite[Theorem~2.1]{SturmfelsSullivant2008}. In particular, the ``Lift and Quad construction'' preserves degree two. Since the tree case is known, the cycle is the only nontrivial building block needed for connected ring graphs.

Nagel and Petrovi\'c claimed that the cut ideal of every ring graph admits a squarefree quadratic Gr\"obner basis \cite[Theorem~6.2]{NagelPetrovic2009}. 
For connected ring graphs, their argument reduces the statement,
via the clique-sum construction, to the corresponding assertions
for trees and cycles.
The tree part is independent of the issue discussed below, but their cycle result \cite[Proposition~3.2]{NagelPetrovic2009} relies on results of Chifman and Petrovi\'c concerning the toric ideal $I_m$ of the $\mathbb Z_2$ group-based model on the claw tree $K_{1,m}$, which coincides with the cut ideal of the cycle of length $m+1$. 
Following Sakamoto's notation, let $Q_m$ denote the recursively
defined set of quadratic binomials considered by Chifman and Petrovi\'c.
They claimed that $Q_m$ generates $I_m$ and, moreover, that $Q_m$
is a lexicographic Gr\"obner basis of $I_m$
\cite[Propositions~2 and~3]{ChifmanPetrovic2007}.
Sakamoto showed that the first of these claims is false for $m\ge5$: the set $Q_m$ does not generate $I_m$ in general. 
Consequently, the proof of the second claim considers a proper subideal of $I_m$, and the Chifman--Petrovi\'c argument used in \cite[Proposition~3.2]{NagelPetrovic2009} does not establish the claimed cycle result for lengths at least $6$ \cite[Section~2]{Sakamoto2021}. 
Consequently, the cycle input needed in the proof of
\cite[Theorem~6.2]{NagelPetrovic2009} remains to be established.
Sakamoto additionally showed that the specific lexicographic order used by Chifman and Petrovi\'c does not yield a quadratic reduced Gr\"obner basis for $m\ge5$. This does not by itself give a counterexample to the existence statement. Sakamoto subsequently constructed lexicographic quadratic Gr\"obner bases for cycles of length at most $7$ and left the general existence problem open \cite{Sakamoto2021}. 
The present paper resolves that cycle problem by a different weighted order.
Applying this result in the argument of Nagel and Petrovi\'c
\cite[Theorem~6.2]{NagelPetrovic2009} then gives the corresponding
conclusion for connected ring graphs with at least one edge.

Let $C_n$ be a cycle of length $n\ge3$ having edges $e_1,\dots,e_n$, and let $[n]:=\{1,\dots,n\}$.
A subset of the edge set is a cut if and only if it has even cardinality. Thus the cuts of $C_n$ are naturally indexed by
\[
\En=\{A\subset[n]: |A|\equiv0\pmod2\}
\]
which has cardinality $2^{n-1}$.
Let $R_n=\kk[q_A:A\in\En]$
be a polynomial ring in $2^{n-1}$ variables over a field $\kk$.
The \textit{cut ideal} $I_{C_n}$ of $C_n$
is the kernel of the ring homomorphism
\begin{equation}\label{eq:cutmap}
 \varphi_n:R_n\longrightarrow
 \kk[s_1,\dots,s_n,t_1,\dots,t_n],
 \qquad
q_A\longmapsto
 \prod_{i\in A}s_i\prod_{i\notin A}t_i.
\end{equation}
It is known that $I_{C_n}$ is generated by homogeneous quadratic binomials.
(Note that $I_{C_3} = \{0\}$.)
Set
\[
N=\frac{n(n+1)}2,
 \qquad
M=2N^2+1,
 \qquad
w_i=n-i+1 \quad (1\le i\le n).
\]
For $1\le i<j\le n$, define
\begin{equation}\label{eq:cij}
c_{ij}=w_j(M-w_i) =(n-j+1)\bigl(M-(n-i+1)
\bigr) >0.
\end{equation}
Given $A \in \En$, let
\begin{equation}\label{eq:H}
H(A)=\sum_{\{i,j\}\subset A}c_{ij}, 
\end{equation}
where the empty sum is understood to be zero. Since $c_{ij}>0$, we have $H(A)>0$ for every nonempty $A\in\mathcal E_n$.
We prove the following.

\begin{theorem}\label{thm:main}
For every $n\ge 3$, the cut ideal $I_{C_n}$ has a quadratic Gr\"obner basis with respect to a weight order induced by $H$ in \eqref{eq:H}.
\end{theorem}

Applying \Cref{thm:main} to the argument of Nagel and Petrovi\'c
\cite[Theorem~6.2]{NagelPetrovic2009}
gives the following ring-graph consequence.

\begin{corollary}\label{cor:ringgraph}
If $G$ is a connected ring graph with at least one edge, then its cut ideal $I_G$ admits a quadratic Gr\"obner basis.
\end{corollary}

In particular, the conclusion holds for every connected outerplanar graph with at least one edge.  The proof of Corollary~\ref{cor:ringgraph} is given at the end of \Cref{sec:GB}.

In polyhedral terms, \Cref{thm:main} has the following consequence.

\begin{corollary}\label{cor:triangulation}
The parity polytope $P_n^{\mathrm{even}}$ admits a regular unimodular flag triangulation with respect to the lattice generated by its vertices.
\end{corollary}

Indeed, the exponent vector of $\varphi_n(q_A)$ is
$ (\chi_A,\mathbf 1-\chi_A)\in\{0,1\}^{2n}$,
where $\mathbf 1=(1,\dots,1)$. Hence the configuration defining $I_{C_n}$ is the image of the vertex configuration $\{\chi_A:A\in\En\}$ of $P_n^{\mathrm{even}}$ under the affine lattice embedding
$x\longmapsto (x,\mathbf 1-x)$.
Thus $I_{C_n}$ is the toric ideal of a configuration affinely lattice-isomorphic to the vertex configuration of $P_n^{\mathrm{even}}$. 
By \Cref{thm:main} and
\cite[Proposition~10]{MatsudaOhsugiShibata2019},
the corresponding initial ideal is generated by \textit{squarefree}
quadratic monomials.
It therefore gives a regular unimodular
flag triangulation; see \cite[Chapter~8]{Sturmfels1996}.

The proof of \Cref{thm:main} is based on a parity-adjusted sorting operation on even subsets. The weight in \eqref{eq:H} is chosen so that, among quadratic monomials with the same image under $\varphi_n$, the even-sorted monomial is the unique one of minimum weight. To prove that the resulting quadratic relations form a Gr\"obner basis, it then suffices to understand degree three. We show that, among cubic monomials with a fixed image under $\varphi_n$, there is a unique monomial whose three quadratic divisors are even-sorted. Consequently, the two monomials of every nontrivial $S$-polynomial reduce to the same normal form, and Buchberger's criterion \cite[Theorem~1.29]{HerzogHibiOhsugi2018} applies.

The paper is organized as follows. 
Section~2 introduces the parity-adjusted sorting operation on pairs of even subsets and develops the corresponding structure of quadratic fibers. 
Section~3 constructs the explicit weight order and proves that the
even-sorted monomial is the unique minimizer in every quadratic fiber.
Section~4 establishes the uniqueness of the pairwise even-sorted monomial in each cubic fiber. 
Finally, Section~5 proves \Cref{thm:main} using Buchberger's criterion and gives the consequences for ring graphs and Koszulness.

\subsection*{Acknowledgment}
This work was supported by 
JSPS KAKENHI 24K00534.

\section{Parity-adjusted sorting of two even subsets}

Our construction is a parity-adjusted version of the usual sorting
operation for monomials.
See, for example,
\cite[Theorem~14.2]{Sturmfels1996}.
When the usual sorting produces two even subsets, our operation agrees
with it; otherwise, the largest element of the symmetric difference is
shifted from one subset to the other in order to preserve even cardinality.

Given $A\subset[n]$, let $\chi_A\in\{0,1\}^n$ denote its
characteristic vector, and write $\chi_A(i)$ for its $i$th coordinate.
Thus
\[
 \chi_A(i)=
 \begin{cases}
 1,& i\in A,\\
 0,& i\notin A
 \end{cases}
 \qquad(1\le i\le n).
\]
Given a monomial $m=q_{A_1}\cdots q_{A_d} \in R_n$ of degree $d$, let
\[
\rho(m)=\chi_{A_1}+\cdots+\chi_{A_d}
\in\{0,1,\dots,d\}^n.
\]
For $d\ge1$ and $u\in\{0,1,\dots,d\}^n$, define
\[
\mathcal F_u^{(d)}
=
\left\{
q_{A_1}\cdots q_{A_d} \in R_n
:
A_1,\ldots,A_d\in\En,\ 
\rho(q_{A_1}\cdots q_{A_d})=u
\right\}.
\]
If $\mathcal F_u^{(d)}\neq\emptyset$, then we call
$\mathcal F_u^{(d)}$ the \emph{degree-$d$ fiber of $\varphi_n$ over $u$}.
A degree-$2$ fiber (resp.~degree-$3$ fiber) is called a
\emph{quadratic fiber} (resp.~\emph{cubic fiber}).
If $h:\En\to\mathbb R$ is a weight function, the $h$-weight of a monomial $m=q_{A_1}\cdots q_{A_d}$ is 
\[
\wt_h(m):=h(A_1)+\cdots+h(A_d).
\]
For a nonempty quadratic fiber $\mathcal F_u^{(2)}$, a monomial $m\in\mathcal F_u^{(2)}$ is called a \emph{quadratic minimizer} (with respect to $h$) if
\[
 \wt_h(m)
=
 \min\{\wt_h(m'):m'\in\mathcal F_u^{(2)}\}.
\]
If there is exactly one such monomial, it is called the \emph{unique quadratic minimizer} of $\mathcal F_u^{(2)}$.
For monomials $m,m'\in R_n$ of the same degree $d$, we have
\[
 \varphi_n(m)=\varphi_n(m')
 \quad\Longleftrightarrow\quad
 \rho(m)=\rho(m')
\]
since the exponents of $s_i$ and $t_i$ in $\varphi_n(m)$ are
$(\rho(m))_i$ and $d-(\rho(m))_i$, respectively.
Thus
$\mathcal F_u^{(d)}$ is precisely a fiber of
$\varphi_n$ restricted to monomials of degree $d$. 
In particular, a
quadratic fiber is determined by $\chi_A+\chi_B$, and a cubic fiber is determined by
$\chi_A+\chi_B+\chi_C$.

\begin{definition} \label{def:evensort}
Fix $A,B\in\En$. 
Let
\[
I=A\cap B,
 \qquad
A\triangle B=\{d_1<d_2<\cdots<d_{2m}\},
 \qquad
p\in\{0,1\} \mbox{ with } p\equiv |I|\pmod2.
\]
Since $A$ and $B$ are both even,
the symmetric difference $A \triangle B$ has even (possibly zero) cardinality.
\begin{enumerate}
    \item 
If $m\equiv p\pmod2$, define
\begin{align*}
 A^\sharp&=I\cup\{d_{2r-1}:1\le r\le m\},\\
 B^\sharp&=I\cup\{d_{2r}:1\le r\le m\}.
\end{align*}
\item
If $m\not\equiv p\pmod2$ (which necessarily implies $m\ge1$), define
\begin{align*}
 A^\sharp&=I\cup\{d_{2r-1}:1\le r\le m\}\cup\{d_{2m}\},\\
 B^\sharp&=I\cup\{d_{2r}:1\le r\le m-1\}.
\end{align*}
\end{enumerate}
Here a set indexed by an empty range is understood to be empty. We call the unordered pair $\esort(A,B)=\{A^\sharp,B^\sharp\}$ the \emph{even-sorted pair} associated with $A,B$. We say that $\{A,B\}$ is even-sorted if $\{A,B\}=\esort(A,B)$.
\end{definition}

\begin{lemma}\label{lem:sortingfiber}
Let $A,B\in\En$, and let
$\{A^\sharp,B^\sharp\}=\esort(A,B)$. Then
$A^\sharp,B^\sharp\in\En$, and 
\[
 \chi_A+\chi_B=\chi_{A^\sharp}+\chi_{B^\sharp}.
\]
Consequently, $q_Aq_B$ and
$q_{A^\sharp}q_{B^\sharp}$ belong to the same quadratic fiber.
In particular,
\[
 q_Aq_B-q_{A^\sharp}q_{B^\sharp}\in I_{C_n}.
\]
\end{lemma}

\begin{proof}
Work with the same notation as in Definition~\ref{def:evensort}.
\begin{itemize}
    \item 
Suppose that $m\equiv p\pmod2$. 
Then
\[
 |A^\sharp|
 =
 |B^\sharp|
 =
 |I|+m
 \equiv
 p+p
 \equiv0
 \pmod2.
\]
Thus $A^\sharp,B^\sharp\in\En$.
\item 
Suppose that $m\not\equiv p\pmod2$, i.e., $m\equiv p+1\pmod2$.
Then
\[
 \begin{aligned}
 |A^\sharp|
 &=|I|+m+1\equiv p+(p+1)+1\equiv0\pmod2,\\
 |B^\sharp|
 &=|I|+m-1\equiv p+(p+1)-1\equiv0\pmod2.
 \end{aligned}
\]
Thus $A^\sharp,B^\sharp\in\En$.
\end{itemize}
In either case, the construction gives
\[
 A^\sharp\cap B^\sharp=I=A\cap B
\]
and
\[
 A^\sharp\triangle B^\sharp
 =
 A\triangle B.
\]
For arbitrary subsets $X,Y\subset[n]$, one has
\[
 \chi_X+\chi_Y
 =
 2\chi_{X\cap Y}+\chi_{X\triangle Y}.
\]
Thus we have
\[
 \chi_{A^\sharp}+\chi_{B^\sharp}
 =
 2\chi_{A^\sharp\cap B^\sharp}
 +\chi_{A^\sharp\triangle B^\sharp}
 =
 2\chi_{A\cap B}
 +\chi_{A\triangle B}
 =
 \chi_A+\chi_B.
\]
Equivalently,
\[
 \rho(q_Aq_B)
 =
 \chi_A+\chi_B
 =
 \chi_{A^\sharp}+\chi_{B^\sharp}
 =
 \rho(q_{A^\sharp}q_{B^\sharp}).
\]
Hence $q_Aq_B$ and $q_{A^\sharp}q_{B^\sharp}$ belong to the
same quadratic fiber.  By the description of fibers above,
\[
 \varphi_n(q_Aq_B)
 =
 \varphi_n(q_{A^\sharp}q_{B^\sharp}),
\]
and therefore
\[
 q_Aq_B-q_{A^\sharp}q_{B^\sharp}
 \in\ker\varphi_n=I_{C_n},
\]
as desired.
\end{proof}

Given \(U\subset[n]\), define
$$
F_k(U)=|U\cap[k]|
\qquad (0\le k\le n),
$$
where \(F_0(U)=0\). We regard the sequence
$$
F_0(U),F_1(U),\ldots,F_n(U)
$$
as the \emph{prefix-count path} of \(U\). If \(U,V\in\En\) form an even-sorted pair, then one of their prefix-count paths weakly dominates the other.

\begin{lemma} \label{lem:pairshape}
Let $U,V\in\En$ be even-sorted. 
After possibly interchanging $U$ and $V$, one has
\[
F_k(U)\ge F_k(V)\qquad(1\le k\le n).
\]
Moreover,
$\delta_k:=F_k(U)-F_k(V)$ satisfies exactly one of the following.
\begin{enumerate}[label=\textup{(\roman*)}]
\item $|U|=|V|$ and $\delta_k\in\{0,1\}$ for all $k$.
\item $|U|=|V|+2$ and $\delta_k\in\{0,1,2\}$ for all $k$. In addition, if
$\delta_k=2$ for some $k$, then $\delta_\ell=2$ for every $\ell\ge k$.
\end{enumerate}
\end{lemma}

\begin{proof}
If $U=V$, then $\delta_k=0$ for every $k$, and the assertion is immediate. Hence we may assume that $U\ne V$. 
Let $U\triangle V=\{d_1<\cdots<d_{2m}\}$.
After interchanging $U$ and $V$ if necessary, we may assume that $d_1$ belongs to $U$.

In the case of Definition~\ref{def:evensort} (1), the elements of the symmetric difference are assigned alternately. The difference of the two prefix paths is $1$ on each interval $d_{2r-1}\le k<d_{2r}$ and $0$ elsewhere. Hence (i) holds.

In the case of Definition~\ref{def:evensort} (2), the same alternating behavior occurs until the last two elements of the symmetric difference. The element $d_{2m-1}$ is assigned to $U$ and then $d_{2m}$ is also assigned to $U$. Thus the path difference becomes $2$ at $d_{2m}$ and remains $2$ thereafter. This gives (ii).
\end{proof}

\section{A weight order for even sorting}

We first introduce an auxiliary weight $\Omega$ for which the minimization argument is more direct.
After proving that even sorting is uniquely minimal for $\Omega$, we show that $\Omega$ and the explicit pairwise weight $H$ from \Cref{thm:main} induce the same comparisons on every fiber. 
Let
\[
N=\sum_{k=1}^n k=\frac{n(n+1)}2,
 \qquad
M=2N^2+1.
\]
For $U\subset[n]$, set
\[
P(U)=\sum_{k=1}^n F_k(U)^2,
 \qquad
L(U)=\sum_{k=1}^nF_k(U),
 \qquad
 \Omega(U)=M \cdot P(U)-L(U)^2.
\]

\begin{lemma}\label{lem:quadraticminimum}
In every quadratic fiber, the even-sorted monomial uniquely minimizes $\Omega(U)+\Omega(V)$.
\end{lemma}

\begin{proof}
Fix a quadratic fiber
\[
 \mathcal{F}^{(2)}_s
=
 \left\{
q_Uq_V : U,V\in\mathcal{E}_n,\;
 \chi_U+\chi_V=s
 \right\},
 \qquad s\in\{0,1,2\}^n.
\]
Set $I_s:=\{i\in[n]:s_i=2\}$ and $D_s:=\{i\in[n]:s_i=1\}$. For every $q_Uq_V\in\mathcal{F}^{(2)}_s$, one has $U\cap V=I_s$ and $U\triangle V=D_s$. 
Let $D_s=\{d_1<d_2<\cdots<d_{2m}\}$.
If $D_s=\emptyset$, then $U=V=I_s$ for every $q_Uq_V\in\mathcal F_s^{(2)}$. Hence the fiber consists of the single monomial $q_{I_s}^2$, and the assertion is immediate. We may therefore assume that $D_s\ne\emptyset$, or equivalently, that $m\ge 1$.
Let $p\in\{0,1\}$ be determined by $p\equiv |I_s|\pmod 2$, and let $q_{U^\sharp}q_{V^\sharp}$ be the even-sorted monomial determined by $I_s$ and $D_s$ as in Definition~\ref{def:evensort}. 
This monomial depends only on the fiber.

For $1\le k\le n$, let
\begin{align}   \label{eq:s and delta}
S_k:=F_k(U)+F_k(V),
\qquad
\Delta_k:=F_k(U)-F_k(V).
\end{align}
Since $\chi_U+\chi_V=s$, we have
\[
S_k
=
F_k(U)+F_k(V)
=
\sum_{i=1}^k\bigl(\chi_U(i)+\chi_V(i)\bigr)
=
\sum_{i=1}^k s_i.
\]
Thus $S_k$ depends only on the fiber
$\mathcal F_s^{(2)}$.
From \eqref{eq:s and delta},
it follows that
\[
F_k(U)^2+F_k(V)^2
=
\frac12\bigl(S_k^2+\Delta_k^2\bigr).
\]
Hence
\[
P(U)+P(V)
=
\frac12\sum_{k=1}^n
\bigl(S_k^2+\Delta_k^2\bigr).
\]
Similarly,
\[
L(U)+L(V)
=
\sum_{k=1}^n S_k,
\qquad
L(U)-L(V)
=
\sum_{k=1}^n \Delta_k,
\]
and hence
\[
L(U)^2+L(V)^2
=
\frac12\left(
\left(\sum_{k=1}^n S_k\right)^2
+
\left(\sum_{k=1}^n\Delta_k\right)^2
\right).
\]
Consequently,
\[
 \Omega(U)+\Omega(V)
= C_s+\frac{M}{2}Q(U,V)-\frac12R(U,V)^2,
\]
where
\[
 C_s:=\frac M2\sum_{k=1}^nS_k^2
 -\frac12\left(\sum_{k=1}^nS_k\right)^2,
\qquad
Q(U,V):=\sum_{k=1}^n\Delta_k^2,
 \qquad
R(U,V):=\sum_{k=1}^n\Delta_k.
\]
Note that $C_s$ depends only on the fiber.
Since $|\Delta_k|\le k$, we have
\[
 |R(U,V)|\le\sum_{k=1}^n|\Delta_k|
 \le\sum_{k=1}^n k=N.
\]

Let $ q_Uq_V,\ q_{U'}q_{V'}\in\mathcal{F}^{(2)}_s$.
\begin{itemize}
    \item 
If $Q(U,V)<Q(U',V')$, then 
$Q(U',V') - Q(U,V) \ge 1$
and hence
we have
\begin{align*}
& \quad ( \Omega(U')+\Omega(V'))-( \Omega(U)+\Omega(V))\\
& =
 \frac{M}{2}\bigl(Q(U',V')-Q(U,V)\bigr)
-
 \frac12\bigl(R(U',V')^2-R(U,V)^2\bigr)\\
& \ge
 \frac{M}{2}-\frac{N^2}{2}>0
\end{align*}
since $|R(U,V)|, |R(U',V')| \le N$.
Thus every minimizer of \(\Omega(U)+\Omega(V)\) must first minimize
\(Q(U,V)\).
\item
If $Q(U,V)=Q(U',V')$, then 
\[
( \Omega(U')+\Omega(V'))-( \Omega(U)+\Omega(V))
=-
 \frac12\bigl(R(U',V')^2-R(U,V)^2\bigr).
\]
Thus among the minimizers of \(Q\), minimizing
\(\Omega(U)+\Omega(V)\) is equivalent to maximizing
$|R(U,V)|$.
\end{itemize}

We now determine the minimizers of \(Q\). 
Recall that $D_s=U\triangle V=\{d_1<\cdots<d_{2m}\}$.
Note that $\Delta_k$ can change only when $k\in D_s$.
Hence $\Delta_k$ is constant between two consecutive elements of $D_s$.
Let $\{U^\sharp,V^\sharp\}=\esort(U,V)$ be the even-sorted pair determined by the fiber. 
We may assume that \(d_1\in U^\sharp\).
Let $ \Delta_k^\sharp
:= F_k(U^\sharp)-F_k(V^\sharp)$.

\medskip

\noindent
{\bf Case 1} (\(m\equiv p\pmod2\)). Then $|U^\sharp|=|V^\sharp|$
and $\Delta_k^\sharp$
is given explicitly by
\[
 \Delta_k^\sharp=
 \begin{cases}
1, & d_{2r-1}\le k<d_{2r}
    \text{ for some }r,\\
0, & \text{otherwise}.
 \end{cases}
\]
In particular, $(\Delta_k^\sharp)^2=\Delta_k^\sharp $.
Let $\ell_r:=d_{2r}-d_{2r-1}>0$. 
It then follows that
\[
Q(U^\sharp,V^\sharp) =
R(U^\sharp,V^\sharp) =
 \sum_{r=1}^m\ell_r.
\]

Let $q_Uq_V\in\mathcal{F}^{(2)}_s$.
If $d_{2r-1}\le k<d_{2r}$, then exactly $2r-1$ elements of
$D_s=U\triangle V$ belong to $[k]$. 
Hence
\[
\Delta_k
=
|(U\setminus V)\cap[k]|
-
|(V\setminus U)\cap[k]|
\equiv
|(U\setminus V)\cap[k]|
+
|(V\setminus U)\cap[k]|
= 2r - 1
\pmod 2
,\]
and therefore $|\Delta_k|\ge1$.
Consequently,
\[
Q(U,V) =
 \sum_{k=1}^n\Delta_k^2
 \ge
 \sum_{r=1}^m
 \sum_{k=d_{2r-1}}^{d_{2r}-1}1
=
 \sum_{r=1}^m\ell_r
= Q(U^\sharp,V^\sharp).
\]
Hence the even-sorted pair minimizes $Q$.

Moreover, equality
\[
Q(U,V)=Q(U^\sharp,V^\sharp)
\]
can hold if and only if 
\[
 |\Delta_k|=
 \begin{cases}
1, & d_{2r-1}\le k<d_{2r}
    \text{ for some }r,\\
0, & \text{otherwise}.
 \end{cases}
\]
In such a case,
there are signs $\sigma_r\in\{\pm1\}$
such that
\[
 \Delta_k=\sigma_r
 \qquad
(d_{2r-1}\le k<d_{2r}).
\]
Hence every \(Q\)-minimizer satisfies
\[
R(U,V) =
 \sum_{r=1}^m\sigma_r\ell_r.
\]
By the triangle inequality,
\[
|R(U,V)|
 \le
 \sum_{r=1}^m\ell_r
= |R(U^\sharp,V^\sharp)|.
\]
Since every \(\ell_r\) is positive, equality holds if and only if
\[
 \sigma_1=\sigma_2=\cdots=\sigma_m.
\]
The two possible common signs correspond precisely to interchanging
\(U^\sharp\) and \(V^\sharp\).
Among the $Q$-minimizers, $|R|$ is therefore maximal only for $\{U^\sharp,V^\sharp\}$. Hence $q_{U^\sharp}q_{V^\sharp}$ is the unique quadratic minimizer in Case~1.

\medskip

\noindent
{\bf Case 2} (\(m\not\equiv p\pmod2\)). 
Then
\begin{align*}
U^\sharp
&=I_s\cup\{d_{2r-1}:1\le r\le m\}\cup\{d_{2m}\},\\
V^\sharp
&=I_s\cup\{d_{2r}:1\le r\le m-1\}.
\end{align*}
Hence $\Delta_k^\sharp$
is given by
\[
 \Delta_k^\sharp=
 \begin{cases}
1, & d_{2r-1}\le k<d_{2r}
    \text{ for some }r,\\
0, & d_{2r}\le k<d_{2r+1}
    \text{ for some }r<m,\\
2, & d_{2m}\le k\le n.
 \end{cases}
\]
Let $\ell_r:=d_{2r}-d_{2r-1}>0$ and $\tau:=n-d_{2m}+1>0$. It follows that
\[
Q(U^\sharp,V^\sharp) =
 \sum_{r=1}^m\ell_r+4\tau,
 \qquad
R(U^\sharp,V^\sharp) =
 \sum_{r=1}^m\ell_r+2\tau.
\]

Let $q_Uq_V\in\mathcal{F}^{(2)}_s$.
From the same discussion in Case 1,
on every interval
$d_{2r-1}\le k<d_{2r}$,
we have $|\Delta_k|\ge1$.
There is an additional restriction after the last element
\(d_{2m}\).  Since both \(|U|\) and \(|V|\) are even,
\[
|U\setminus I_s|\equiv |V\setminus I_s|
 \equiv |I_s|
 \equiv p
 \pmod2.
\]
On the other hand,
\[
|U\setminus I_s|+|V\setminus I_s|=2m.
\]
Moreover,
\[
 \Delta_n
=
F_n(U) - F_n(V)
= |U|-|V| 
= |U \setminus I_s|-|V \setminus I_s| 
\equiv |U \setminus I_s|+|V \setminus I_s| 
=2m
\equiv0
\pmod 2.
\]
If \(\Delta_n=0\), then
\[
|U\setminus I_s|=m,
\]
which would imply
\[
m\equiv p\pmod2,
\]
contrary to the assumption of Case~2. 
Since $\Delta_n$ is even,
we have $|\Delta_n|\ge2$.
Since no element of the symmetric difference occurs after \(d_{2m}\), the value of \(\Delta_k\) is constant for \(k\ge d_{2m}\). Thus
\[
|\Delta_k|\ge2
 \qquad(d_{2m}\le k\le n).
\]
Consequently,
\[
Q(U,V) =
 \sum_{k=1}^n\Delta_k^2
\ge
 \sum_{r=1}^m
 \sum_{k=d_{2r-1}}^{d_{2r}-1}1
+
 \sum_{k=d_{2m}}^n4
=
 \sum_{r=1}^m\ell_r+4\tau
= Q(U^\sharp,V^\sharp).
\]
Hence the even-sorted pair minimizes $Q$.

Suppose now that equality 
\[
Q(U,V)=Q(U^\sharp,V^\sharp)
\]
holds.
Then 
\[
 |\Delta_k|=
 \begin{cases}
1, & d_{2r-1}\le k<d_{2r}
    \text{ for some }r,\\
0, & d_{2r}\le k<d_{2r+1}
    \text{ for some }r<m,\\
2, & d_{2m}\le k\le n.
 \end{cases}
\]
In such a case,
there are signs $\sigma_r\in\{\pm1\}$
such that
\[
 \Delta_k=\sigma_r
 \qquad
(d_{2r-1}\le k<d_{2r}).
\]
In particular,
$$
 \Delta_{d_{2m}-1}=\sigma_m.
$$
Since exactly one of \(U\) and \(V\) contains \(d_{2m}\), 
we have
$$
 \Delta_{d_{2m}}-\Delta_{d_{2m}-1}\in\{1,-1\}.
$$
On the other hand, equality in the lower bound for \(Q\) requires
$$
 |\Delta_{d_{2m}}|=2.
$$
Since \(\Delta_{d_{2m}-1}=\sigma_m\in\{1,-1\}\), the only possibility is
$$
 \Delta_{d_{2m}}=2\sigma_m.
$$
Finally, since \(d_{2m}\) is the largest element of \(U\triangle V\), \(\Delta_k\) does not change for \(k\ge d_{2m}\). Hence
$$
 \Delta_k=2\sigma_m
 \qquad
 (d_{2m}\le k\le n).
$$
Therefore every \(Q\)-minimizer satisfies
\[
R(U,V) =
 \sum_{r=1}^m\sigma_r\ell_r
+ 2\sigma_m\tau.
\]
By the triangle inequality,
\[
|R(U,V)|
 \le
 \sum_{r=1}^m\ell_r+2\tau
= |R(U^\sharp,V^\sharp)|.
\]
Since all \(\ell_r\) and \(\tau\) are positive, equality holds if and only if
\[
 \sigma_1=\sigma_2=\cdots=\sigma_m.
\]
The two possible common signs simply interchange
\(U^\sharp\) and \(V^\sharp\).
Hence the unique unordered pair minimizing $Q$ and then maximizing $|R|$ is
\[
 \left\{
 I_s\cup\{d_{2r-1}:1\le r\le m\}\cup\{d_{2m}\},
 \;
 I_s\cup\{d_{2r}:1\le r\le m-1\}
\right\}.
\]
This is the even-sorted pair, so $q_{U^\sharp}q_{V^\sharp}$ is the unique quadratic minimizer in Case~2.

\medskip

Therefore the assertion follows in both cases.
\end{proof}

\begin{proposition}\label{prop:Hcomparison}
On every fiber, the weights $\Omega$ and $H$ induce the same strict comparisons. In particular, in every quadratic fiber the even-sorted monomial is the unique minimizer of $H(U)+H(V)$.
\end{proposition}

\begin{proof}
Note that, for $U\in\En$, we have
\[
F_k(U)= |U \cap [k]|=\sum_{i=1}^k \chi_U(i).
\]
Let $w_i=n-i+1$. 
Since \(\chi_U(i)^2=\chi_U(i)\), we have the following.
\[
 \begin{aligned}
 L(U)&=\sum_{k=1}^n \sum_{i=1}^k  \chi_U(i)=\sum_{i=1}^n w_i\chi_U(i),\\
L(U)^2
&=\left(\sum_{i=1}^n w_i\chi_U(i)\right)^2
=\sum_{i=1}^n w_i^2\chi_U(i)
 +2\sum_{1\le i<j\le n}
 w_iw_j\chi_U(i)\chi_U(j),\\
P(U)&=\sum_{k=1}^n \left(\sum_{i=1}^k  \chi_U(i)\right)^2 
=\sum_{k=1}^n \left(
\sum_{i=1}^k \chi_U(i)+2\sum_{1\le i<j \le k}\chi_U(i)\chi_U(j)
\right)\\
&=\sum_{i=1}^n w_i\chi_U(i)+2\sum_{1\le i<j \le n}w_j\chi_U(i)\chi_U(j).
 \end{aligned}
\]
Thus
$$
\begin{aligned}
\Omega(U)
&=M \cdot P(U)-L(U)^2\\
&=\sum_{i=1}^n w_i(M-w_i)\chi_U(i) +2\sum_{1\le i<j\le n}
 w_j(M-w_i)\chi_U(i)\chi_U(j).
\end{aligned}
$$
Since \(c_{ij}=w_j(M-w_i)\), we obtain
\begin{equation}\label{eq:OmegaHdecomp}
\Omega(U)=\lambda(U)+2H(U),
\end{equation}
where
\[\lambda(U):=\sum_{i=1}^n \beta_i \chi_U(i)
\mbox{ with } \beta_i:=  w_i(M-w_i).\]

Let $q_{U_1}\cdots q_{U_d}$ and $q_{V_1}\cdots q_{V_d}$ lie in the same degree-$d$ fiber of $\varphi_n$. 
Then 
\[\sum_{a=1}^d\chi_{U_a}(i)=\sum_{a=1}^d\chi_{V_a}(i)\]
for $1\le i\le n$.
Hence
\[
\sum_{a=1}^d\lambda(U_a)
= \sum_{a=1}^d \sum_{i=1}^n \beta_i \chi_{U_a}(i)
=  \sum_{i=1}^n \beta_i \sum_{a=1}^d \chi_{U_a}(i)
=  \sum_{i=1}^n \beta_i \sum_{a=1}^d \chi_{V_a}(i)
= \sum_{a=1}^d \sum_{i=1}^n \beta_i\chi_{V_a}(i)
=\sum_{a=1}^d\lambda(V_a).
\]
Summing \eqref{eq:OmegaHdecomp} over the factors therefore gives
\[
 \sum_{a=1}^d\Omega(U_a)-\sum_{a=1}^d\Omega(V_a)
=2\left(\sum_{a=1}^d H(U_a)-\sum_{a=1}^d H(V_a)\right).
\]
Thus $\Omega$ and $H$ induce the same strict comparisons on every degree-$d$ fiber. The last assertion follows from Lemma~\ref{lem:quadraticminimum}.
\end{proof}

\begin{definition} \label{def:order}
Give the variable $q_A$ the nonnegative integer weight $H(A)$, and write $\wt_H(m)$ for the total $H$-weight of a monomial $m$. Fix any monomial order $\prec_0$ on $R_n$. Define $\prec$ by
\begin{equation}\label{mono_order}
 m\prec m'
 \quad\Longleftrightarrow\quad
 \begin{cases}
 \wt_H(m)<\wt_H(m'),&\text{or}\\
 \wt_H(m)=\wt_H(m')\text{ and }m\prec_0m'.
 \end{cases}    
\end{equation}
Thus a monomial of larger $H$-weight is larger with respect to $\prec$.
\end{definition}

This is a monomial order since the $H$-weight is a nonnegative integer-valued additive function, and ties are refined by the monomial order $\prec_0$.
See, e.g., \cite[Chapters 1 and 3]{Sturmfels1996}.
For every non-even-sorted quadratic monomial $q_Aq_B$, define
\begin{equation}\label{eq:gAB}
g_{A,B}=q_Aq_B-q_{A^\sharp}q_{B^\sharp},
 \qquad
 \{A^\sharp,B^\sharp\}=\esort(A,B).
\end{equation}
By Proposition~\ref{prop:Hcomparison}, the initial monomial is 
\begin{equation}\label{eq:LM}
 \inop_\prec(g_{A,B})=q_Aq_B.
\end{equation}

\section{Uniqueness of the cubic normal form}\label{sec:cubic}

A cubic monomial $q_Xq_Yq_Z$ is called \emph{pairwise even-sorted} if each of $\{X,Y\}$, $\{X,Z\}$, and $\{Y,Z\}$ is even-sorted.

\begin{theorem} \label{thm:cubicunique}
Every cubic fiber contains exactly one pairwise even-sorted monomial.
\end{theorem}

\begin{proof}
We first prove existence. Let $q_Aq_Bq_C$ be any monomial in a fixed cubic fiber. If, say, $\{A,B\}$ is not even-sorted, replace $q_Aq_B$ by the even-sorted monomial $q_{A^\sharp}q_{B^\sharp}$. By Lemma~\ref{lem:sortingfiber}, this replacement does not change the cubic fiber. By Proposition~\ref{prop:Hcomparison}, it strictly decreases the total $H$-weight. Since a cubic fiber is finite, this procedure terminates. At termination all three quadratic divisors are even-sorted, so every cubic fiber contains a pairwise even-sorted monomial.

We prove uniqueness by showing that a pairwise even-sorted monomial is completely determined by its cubic fiber. Let $q_Xq_Yq_Z$ be pairwise even-sorted. By Lemma~\ref{lem:pairshape}, the prefix-count paths of every pair are comparable under pointwise dominance. Since pointwise dominance is transitive, the three paths can be simultaneously ordered. Hence, after relabeling, we may assume
\[
 F_k(X)\ge F_k(Y)\ge F_k(Z)
 \qquad(1\le k\le n).
\]
Let $ X_k:=F_k(X)$, $Y_k:=F_k(Y)$, and $Z_k:=F_k(Z)$.
Since the cardinalities are even and any two members of an even-sorted pair have cardinalities differing by $0$ or $2$, the ordered triple $(|X|,|Y|,|Z|)$ is one of
\[
 (2d,2d,2d),\qquad (2d+2,2d,2d),\qquad (2d+2,2d+2,2d)
\]
for some $d\ge0$. 
The three sums are $6d$, $6d+2$, and $6d+4$,
respectively. 
Since the cubic fiber fixes $|X|+|Y|+|Z|$, it therefore determines which of the three cases occurs.

For each $k$, let
$x_k=\chi_X(k)$, $y_k=\chi_Y(k)$, $z_k=\chi_Z(k)$, 
and let
$$
r_k=x_k + y_k + z_k \in\{0,1,2,3\},
\qquad
T_k=X_k+Y_k+Z_k=\sum_{i=1}^k r_i.
$$
Since the cubic fiber determines each $r_k$, it also determines each $T_k$.
We show in each of the three cases for $(|X|,|Y|,|Z|)$ that these data uniquely determine $X_k$, $Y_k$, and $Z_k$ for every $k$.

\smallskip
\noindent
{\bf Case 1.} $(|X|,|Y|,|Z|)=(2d,2d,2d)$.
Since \(X\) and \(Z\) have the same cardinality, Lemma~\ref{lem:pairshape}(i) gives
$$
0\le X_k-Z_k\le1
\qquad(1\le k\le n).
$$
Together with $X_k\ge Y_k\ge Z_k$, this shows that $(X_k,Y_k,Z_k)$ is an ordered triple of integers with sum $T_k$ whose largest and smallest entries differ by at most $1$.
Writing $T_k=3u_k+v_k$ with
$u_k \in \mathbb{Z}$ and $v_k\in\{0,1,2\}$, we have
\[
 (X_k,Y_k,Z_k)=
 \begin{cases}
 (u_k,u_k,u_k),&v_k=0,\\
 (u_k+1,u_k,u_k),&v_k=1,\\
 (u_k+1,u_k+1,u_k),&v_k=2.
 \end{cases}
\]
Thus all three prefix paths are uniquely determined by the fiber.

\smallskip

\noindent
{\bf Case 2.}
$(|X|,|Y|,|Z|)=(2d+2,2d,2d)$.
Let $a_k=X_k-Y_k$, $b_k=Y_k-Z_k$, and $s_k=(a_k,b_k)$.
Since
$$
|Y|-|Z|=0,\qquad |X|-|Y|=2,\qquad |X|-|Z|=2,
$$
Lemma~\ref{lem:pairshape} gives
$$
b_k\in\{0,1\},\qquad
a_k\in\{0,1,2\},\qquad
a_k+b_k=X_k-Z_k\in\{0,1,2\}.
$$
Therefore
\[
 s_k\in\{00,10,01,11,20\}.
\]
Here, we abbreviate a state $(a,b)$ as $ab$.
Moreover, the terminal part of Lemma~\ref{lem:pairshape}(ii), applied to $(X,Y)$ and $(X,Z)$, gives
\begin{align} \label{T_rule 1}
     a_k=2\Longrightarrow a_\ell=2\ (\ell\ge k),
 \qquad
 a_k+b_k=2\Longrightarrow a_\ell+b_\ell=2\ (\ell\ge k).
\end{align}
Since $|X|-|Y|=2$ and $|Y|-|Z|=0$, the final state is $s_n=(2,0)=20$. 
We now recover the states from right to left.
It follows that
\begin{equation}\label{eq:gaptrans}
\begin{aligned}
 a_k-a_{k-1}&= (X_k - X_{k-1}) - (Y_k - Y_{k-1}) = x_k-y_k,\\
 b_k-b_{k-1}&= (Y_k - Y_{k-1}) - (Z_k - Z_{k-1})=y_k-z_k.
\end{aligned}
\end{equation}
Recall that $r_k=x_k+y_k+z_k$.
\begin{itemize}
\item
If $r_k=0$ or $3$, then $(x_k,y_k,z_k)$ is $(0,0,0)$ or $(1,1,1)$, respectively, so $s_{k-1}=s_k$ by \eqref{eq:gaptrans}. 

\item 

If \(r_k=1\), then \(k\) belongs to exactly one of \(X,Y,Z\).  
By \eqref{eq:gaptrans}, the three possible $s_{k-1} = (a_{k-1},b_{k-1})$ are
\begin{equation} \label{eq:pred1}
\begin{array}{c|ccc}
 & k\in X & k\in Y & k\in Z\\
\hline
s_{k-1}
&
(a_k-1,b_k)
&
(a_k+1,b_k-1)
&
(a_k,b_k+1)
\end{array}
\end{equation}
for \(s_k=(a_k,b_k)\).  Substituting the five possible values of
\(s_k\) gives
$$
\begin{array}{c|ccc}
s_k
& k\in X
& k\in Y
& k\in Z\\
\hline
00 & (-1,0) & (1,-1) & 01\\
10 & 00 & (2,-1) & \underline{11}\\
01 & (-1,1) & 10 & (0,2)\\
11 & 01 & \underline{20} & (1,2)\\
20 & 10 & (3,-1) & (2,1)
\end{array}
$$
Here entries not belonging to the state set
$
\{00,10,01,11,20\}
$
are impossible.  Two further candidates, namely \(\underline{11}\) in the second
row and \(\underline{20}\) in the fourth row, belong to the state set but are
excluded by \eqref{T_rule 1}.  
Indeed, if \(s_{k-1}=11\), then
\(a_{k-1}+b_{k-1}=2\), so \eqref{T_rule 1} would imply
\(a_k+b_k=2\), contradicting \(s_k=10\).  
Similarly, if
\(s_{k-1}=20\), then \(a_{k-1}=2\), so \eqref{T_rule 1} would
imply \(a_k=2\), contradicting \(s_k=11\).

Thus $s_{k-1}$ is uniquely determined as follows:
$$
\begin{array}{c|ccccc}
s_k&00&10&01&11&20\\
\hline
s_{k-1}&01&00&10&01&10
\end{array}
$$

\item 
If \(r_k=2\), then \(k\) belongs to exactly two of \(X,Y,Z\).  The three possibilities are
$$
k\in X\cap Y,\qquad k\in X\cap Z,\qquad k\in Y\cap Z.
$$
By \eqref{eq:gaptrans}, the three possible $s_{k-1} = (a_{k-1},b_{k-1})$ are
\begin{equation} \label{eq:pred2}
\begin{array}{c|ccc}
 & k\in X\cap Y & k\in X\cap Z & k\in Y\cap Z\\
\hline
s_{k-1}
&
(a_k,b_k-1)
&
(a_k-1,b_k+1)
&
(a_k+1,b_k)
\end{array}
\end{equation}
for \(s_k=(a_k,b_k)\).  Substituting the five possible values of
\(s_k\) gives
$$
\begin{array}{c|ccc}
s_k
& k\in X\cap Y
& k\in X\cap Z
& k\in Y\cap Z\\
\hline
00 & (0,-1) & (-1,1) & 10\\
10 & (1,-1) & 01 & \underline{20}\\
01 & 00 & (-1,2) & \underline{11}\\
11 & 10 & (0,2) & (2,1)\\
20 & (2,-1) & 11 & (3,0)
\end{array}
$$
Entries not belonging to the state set
$
\{00,10,01,11,20\}
$
are impossible.  There are two additional candidates that belong to
the state set but are excluded by \eqref{T_rule 1}.  If \(s_k=10\)
and \(s_{k-1}=20\), then \(a_{k-1}=2\), so \eqref{T_rule 1}
would imply \(a_k=2\), a contradiction.  If \(s_k=01\) and
\(s_{k-1}=11\), then \(a_{k-1}+b_{k-1}=2\), so \eqref{T_rule 1} would imply \(a_k+b_k=2\), again a contradiction.
Thus $s_{k-1}$ is uniquely determined as follows:
$$
\begin{array}{c|ccccc}
s_k&00&10&01&11&20\\
\hline
s_{k-1}&10&01&00&10&11
\end{array}
$$
\end{itemize}

Thus we obtain the reverse transition table for $s_{k-1}$.
\begin{center}
\begin{tabular}{c|ccc}
\toprule
$s_k$ & $r_k=0\text{ or }3$ & $r_k=1$ & $r_k=2$\\
\midrule
$00$ & $00$ & $01$ & $10$\\
$10$ & $10$ & $00$ & $01$\\
$01$ & $01$ & $10$ & $00$\\
$11$ & $11$ & $01$ & $10$\\
$20$ & $20$ & $10$ & $11$\\
\bottomrule
\end{tabular}
\end{center}

Starting from the known final state $s_n=20$ and using the fiber data $r_n,r_{n-1},\ldots,r_1$, the table determines successively $s_{n-1},s_{n-2},\ldots,s_0$. Thus every $a_k$ and $b_k$ is determined. Since $T_k$ is also fixed by the fiber, the three prefix counts are recovered from
\begin{equation}\label{eq:recover}
 Y_k=\frac{T_k-a_k+b_k}{3},\qquad
 X_k=Y_k+a_k,\qquad
 Z_k=Y_k-b_k.
\end{equation}
Finally, with $X_0=Y_0=Z_0=0$, the successive differences $\chi_X(k)=X_k-X_{k-1}$, $\chi_Y(k)=Y_k-Y_{k-1}$, and $\chi_Z(k)=Z_k-Z_{k-1}$ recover the subsets $X,Y,Z$ uniquely.

\smallskip
\noindent
{\bf Case 3.} $(|X|,|Y|,|Z|)=(2d+2,2d+2,2d)$.
We use the same notation $a_k=X_k-Y_k$, $b_k=Y_k-Z_k$, and $s_k=(a_k,b_k)$.
Since
$$
 |X|-|Y|=0,\qquad |Y|-|Z|=2,\qquad |X|-|Z|=2,
$$
Lemma~\ref{lem:pairshape} gives
\[
 a_k\in\{0,1\},\qquad b_k\in\{0,1,2\},\qquad a_k+b_k = X_k -Z_k \in\{0,1,2\},
\]
and hence
\[
 s_k\in\{00,10,01,11,02\}.
\]
The terminal part of Lemma~\ref{lem:pairshape}(ii), applied to $(Y,Z)$ and $(X,Z)$, gives
\begin{align} \label{T_rule 2}
 b_k=2\Longrightarrow b_\ell=2\ (\ell\ge k),
 \qquad
 a_k+b_k=2\Longrightarrow a_\ell+b_\ell=2\ (\ell\ge k),
    \end{align}
and the final state is $s_n=02$.
The formulas \eqref{eq:pred1} and \eqref{eq:pred2} remain valid, since they depend only on \eqref{eq:gaptrans}. 
\begin{itemize}
\item 
As in Case 2, if $r_k=0$ or $3$, then $(x_k,y_k,z_k)$ is $(0,0,0)$ or $(1,1,1)$, respectively, so $s_{k-1}=s_k$. 
    \item 
If \(r_k=1\), then \(k\) belongs to exactly one of \(X,Y,Z\).  As in Case~2, the three possible $s_{k-1}=(a_{k-1},b_{k-1})$ are
$$
\begin{array}{c|ccc}
 & k\in X & k\in Y & k\in Z\\
\hline
s_{k-1}
&
(a_k-1,b_k)
&
(a_k+1,b_k-1)
&
(a_k,b_k+1)
\end{array}
$$
for \(s_k=(a_k,b_k)\).  Substituting the five possible values of \(s_k\) gives
$$
\begin{array}{c|ccc}
s_k
& k\in X
& k\in Y
& k\in Z\\
\hline
00 & (-1,0) & (1,-1) & 01\\
10 & 00 & (2,-1) & \underline{11}\\
01 & (-1,1) & 10 & \underline{02}\\
11 & 01 & (2,0) & (1,2)\\
02 & (-1,2) & 11 & (0,3)
\end{array}
$$
Entries not belonging to the state set
$
\{00,10,01,11,02\}
$
are impossible.
In addition, if \(s_k=10\) and \(s_{k-1}=11\), then
$
a_{k-1}+b_{k-1}=2,
$
so \eqref{T_rule 2} would imply \(a_k+b_k=2\), contrary to \(s_k=10\).  Similarly, if \(s_k=01\) and \(s_{k-1}=02\), then \(b_{k-1}=2\), so \eqref{T_rule 2} would imply \(b_k=2\), contrary to \(s_k=01\).
Thus $s_{k-1}$ is uniquely determined as follows:
$$
\begin{array}{c|ccccc}
s_k&00&10&01&11&02\\
\hline
s_{k-1}&01&00&10&01&11
\end{array}
$$

    \item 
    If \(r_k=2\), then \(k\) belongs to exactly two of \(X,Y,Z\).  The three possibilities give
$$
\begin{array}{c|ccc}
 & k\in X\cap Y & k\in X\cap Z & k\in Y\cap Z\\
\hline
s_{k-1}
&
(a_k,b_k-1)
&
(a_k-1,b_k+1)
&
(a_k+1,b_k)
\end{array}
$$
for \(s_k=(a_k,b_k)\).  Hence
$$
\begin{array}{c|ccc}
s_k
& k\in X\cap Y
& k\in X\cap Z
& k\in Y\cap Z\\
\hline
00 & (0,-1) & (-1,1) & 10\\
10 & (1,-1) & 01 & (2,0)\\
01 & 00 & (-1,2) & \underline{11}\\
11 & 10 & \underline{02} & (2,1)\\
02 & 01 & (-1,3) & (1,2)
\end{array}
$$
Again, entries outside
$
\{00,10,01,11,02\}
$
are impossible.  
If \(s_k=01\) and \(s_{k-1}=11\), then
$
a_{k-1}+b_{k-1}=2,
$
so \eqref{T_rule 2} would force \(a_k+b_k=2\), contrary to \(s_k=01\).  If \(s_k=11\) and \(s_{k-1}=02\), then \(b_{k-1}=2\), so \eqref{T_rule 2} would force \(b_k=2\), contrary to \(s_k=11\).
Thus $s_{k-1}$ is uniquely determined as follows:
$$
\begin{array}{c|ccccc}
s_k&00&10&01&11&02\\
\hline
s_{k-1}&10&01&00&10&01
\end{array}
$$
\end{itemize}
Hence we obtain the reverse transition table for $s_{k-1}$.
\begin{center}
\begin{tabular}{c|ccc}
\toprule
$s_k$ & $r_k=0\text{ or }3$ & $r_k=1$ & $r_k=2$\\
\midrule
$00$ & $00$ & $01$ & $10$\\
$10$ & $10$ & $00$ & $01$\\
$01$ & $01$ & $10$ & $00$\\
$11$ & $11$ & $01$ & $10$\\
$02$ & $02$ & $11$ & $01$\\
\bottomrule
\end{tabular}
\end{center}
Thus, starting from $s_n=02$, the fiber data determine all states uniquely in reverse. Equation \eqref{eq:recover} then determines the three prefix paths, and their successive differences determine $X,Y,Z$.

\smallskip

In Cases 1--3, a pairwise even-sorted monomial in the given
cubic fiber is uniquely determined. Together with existence,
this proves the theorem.
\end{proof}

\section{The quadratic Gr\"obner basis}\label{sec:GB}

Let $G_n$ be the set of binomials in \eqref{eq:gAB}, one for each
non-even-sorted quadratic monomial; that is,
\[
G_n:=\{
q_Aq_B-q_{A^\sharp}q_{B^\sharp} : 
A, B \in \En, 
  \{A^\sharp,B^\sharp\}=\esort(A,B) \ne \{A,B\}
 \}.
\]
From Lemma~\ref{lem:sortingfiber}, we have
$G_n  \subset I_{C_n}$.

\begin{lemma}\label{lem:Ggenerates}
The set $G_n$ generates $I_{C_n}$.
\end{lemma}

\begin{proof}
Since $C_n$ is $K_4$-minor-free, $I_{C_n}$ is generated by
quadratic binomials by \cite[Corollary~2.8]{Engstrom2011}.
Thus it is enough to consider a quadratic binomial
$m_1-m_2\in I_{C_n}$.

Since $\rho(m_1)=\rho(m_2)$, the monomials $m_1$ and $m_2$
lie in the same quadratic fiber. Let $m_0$ be the unique
even-sorted monomial in this fiber. For $i=1,2$, if
$m_i\ne m_0$, then $m_i$ is non-even-sorted and
$m_i-m_0\in G_n$. Hence
\[
m_1-m_2=(m_1-m_0)-(m_2-m_0),
\]
where a term on the right is understood to be zero if
$m_i=m_0$. Thus $m_1-m_2$ belongs to the ideal generated by
$G_n$.
\end{proof}

\begin{theorem}\label{thm:GB}
The set $G_n$ is a quadratic Gr\"obner basis of $I_{C_n}$ with respect to the monomial order $\prec$ in \eqref{mono_order}.
\end{theorem}

\begin{proof}
By \eqref{eq:LM}, the initial monomials of the elements of $G_n$ are
precisely the non-even-sorted quadratic monomials, and
distinct elements of $G_n$ have distinct initial monomials.
If $q_Aq_B$ is non-even-sorted, then $A\ne B$.
Hence every such monomial is squarefree.

We verify Buchberger's criterion \cite[Theorem~1.29]{HerzogHibiOhsugi2018}.  Let $f,g\in G_n$ be distinct.  If
$\inop_\prec(f)$ and $\inop_\prec(g)$ are relatively prime, then
$S(f,g)$ reduces to zero by the product criterion \cite[Lemma~1.27]{HerzogHibiOhsugi2018}.

Suppose that they are not relatively prime.  Since they are distinct
squarefree quadratic monomials, they have exactly one common variable.
Let
\[
 \inop_\prec(f)=q_Aq_B,\qquad
 \inop_\prec(g)=q_Aq_C,
\]
and let
\[
 f=q_Aq_B-m_f,\qquad
 g=q_Aq_C-m_g,
\]
where $m_f$ and $m_g$ are the corresponding even-sorted quadratic
monomials.  Then
\[
 S(f,g)=q_Cf-q_Bg
       =q_Bm_g-q_Cm_f.
\]
Since $f$ and $g$ belong to the toric ideal, the two cubic monomials
$q_Bm_g$ and $q_Cm_f$ have the same image under $\varphi_n$.
Therefore they lie in the same cubic fiber.

Every reduction by $G_n$ replaces a non-even-sorted quadratic factor
by its even-sorted factor.  By Proposition~\ref{prop:Hcomparison}, this strictly
decreases the total $H$-weight.  Hence every reduction sequence
starting from either $q_Bm_g$ or $q_Cm_f$ terminates.  A terminal
cubic monomial has no non-even-sorted quadratic divisor, and is
therefore pairwise even-sorted.
By \Cref{thm:cubicunique}, the cubic fiber under consideration
contains exactly one pairwise even-sorted monomial. 
Therefore, both monomials reduce to this same monomial, so $S(f,g)$ reduces to zero.

Thus every $S$-polynomial of elements of $G_n$ reduces to zero.  
Buchberger's
criterion shows that $G_n$ is a Gr\"obner basis of the ideal it generates.
Finally, Lemma~\ref{lem:Ggenerates} shows that this ideal is
$I_{C_n}$.
\end{proof}

\begin{proof}[Proof of \Cref{thm:main}]
The weight is given by \cref{eq:cij,eq:H}, and the assertion follows from \Cref{thm:GB}.
\end{proof}

\begin{proof}[Proof of Corollary~\ref{cor:ringgraph}]
The proof follows the argument of Nagel and Petrovi\'c
\cite[Theorem~6.2]{NagelPetrovic2009}.
A connected ring graph can be obtained from trees and cycles by repeated
clique sums along a vertex or an edge. For trees with at least two edges,
the required quadratic Gr\"obner basis is given by
\cite[Corollary~4.3]{NagelPetrovic2009}, while the one-edge case is trivial.
\Cref{thm:main} provides the required quadratic Gr\"obner bases for cycles.
The clique-sum theorem of Sturmfels and Sullivant
\cite[Theorem~2.1]{SturmfelsSullivant2008}
then yields a quadratic Gr\"obner basis at each step.
Hence $I_G$ has a quadratic Gr\"obner basis.
\end{proof}

\begin{remark}\label{rem:disconnected}
The connectedness hypothesis in Corollary~\ref{cor:ringgraph} concerns the standard cut-ideal presentation by variables indexed by vertex partitions. If $G$ is disconnected, distinct vertex partitions can define the same cut vector, and hence $I_G$ contains linear binomials. For this reason, Corollary~\ref{cor:ringgraph} is stated for connected ring graphs with at least one edge.
\end{remark}

The following is immediate.

\begin{corollary}\label{cor:koszul}
For every $n\ge3$, the cut algebra $R_n/I_{C_n}$ is Koszul.
\end{corollary}

\end{document}